\documentclass{amsart}

\makeatletter
\@namedef{subjclassname@2020}{%
  \textup{2020} Mathematics Subject Classification}
 
\makeatother % Cambiar el ano de la MSC, de 2010 a 2020

\usepackage{amsthm,amssymb,amsfonts,latexsym,mathtools,thmtools,amsmath,lscape}
\usepackage[T1]{fontenc}
\usepackage{tikz-cd} 
\usepackage{enumitem} % Customization of items.
\usepackage{longtable}
\usepackage{multirow}
\usepackage{caption}
\usepackage{mathrsfs} 
\usepackage{hyperref} 
\usepackage{hyperref}
\usepackage{float}
\usepackage[all]{xy}
\usepackage{booktabs}

\hypersetup{
    colorlinks=true,
    linkcolor=blue,
    filecolor=blue,      
    urlcolor=blue,
    linktocpage=true
}

\newtheorem{theorem}{Theorem}[section]

\newtheorem{proposition}[theorem]{Proposition}

\theoremstyle{definition}
\newtheorem{definition}[theorem]{Definition}

\newtheorem{remark}[theorem]{Remark}

\newtheorem{example}[theorem]{Example}

\theoremstyle{remark}

\numberwithin{equation}{section}

\begin{document}

\title[On the smoothness of 3-dimensional skew polynomial rings]{On the smoothness of 3-dimensional \\ skew polynomial rings}

\author{Andr\'es Rubiano}
\address{Universidad Distrital Francisco Jos\'e de Caldas}
\curraddr{Campus Universitario}
\email{aarubianos@udistrital.edu.co} 

\author{Armando Reyes}
\address{Universidad Nacional de Colombia - Sede Bogot\'a}
\curraddr{Campus Universitario}
\email{mareyesv@unal.edu.co}

\thanks{The second author was supported by Faculty of Science, Universidad Nacional de Colombia - Sede Bogot\'a, Colombia [grant number 65488].}

\subjclass[2020]{16E45, 16S30, 16S32, 16S36, 16S38, 16S99, 16T05, 58B34}

\keywords{Differentially smooth algebra, integrable calculus, skew polynomial ring}

\date{}

\dedicatory{Dedicated to Professor Oswaldo Lezama on the Occasion of His 70th Birthday}

\begin{abstract} 

This paper is part of a series of papers in which we have investigated the differential smoothness of families of noncommutative algebras. Here, we consider this topic for the family 3-dimensional skew polynomial rings characterized by Bell and Smith \cite{BellSmith1990}.

\end{abstract}

\maketitle

%\tableofcontents 

\section{Introduction}

Brzezi{\'n}ski and Sitarz \cite{BrzezinskiSitarz2017} defined a notion of smoothness of algebras, termed {\em differential smoothness} due to the use of differential graded algebras of a specified dimension that admits a noncommutative version of the Hodge star isomorphism, which considers the existence of a top form in a differential calculus over an algebra together with a string version of the Poincar\'e duality realized as an
isomorphism between complexes of differential and integral forms. \textquotedblleft The idea behind the {\em differential smoothness} of algebras is rooted in the observation that a classical smooth orientable manifold, in addition to de Rham complex of differential forms, admits also the complex of {\em integral forms} isomorphic to the de Rham complex \cite[Section 4.5]{Manin1997}. The de Rham differential can be understood as a special left connection, while the boundary operator in the complex of integral forms is an example of a {\em right connection}\textquotedblright\ \cite[p. 413]{BrzezinskiSitarz2017}. This notion of smoothness is different and more constructive than other homological smoothnesses considered in the literature (for more details on the subject, see \cite{Brzezinski2008, Brzezinski2014, CuntzQuillen1995, Grothendieck1964, Krahmer2012, Schelter1986, StaffordZhang1994, VandenBergh1998}).

Since its introduction, Brzezi{\'n}ski and several authors (e.g. \cite{Brzezinski2015} - \cite{BrzezinskiSitarz2017}, \cite{DuboisViolette1988, DuboisVioletteKernerMadore1990, Karacuha2015, KaracuhaLomp2014, ReyesSarmiento2022}) have characterized the differential smoothness of different families of noncommutative algebras such as the following: quantum spheres, noncommutative torus, the coordinate algebras of the quantum group $SU_q(2)$, the noncommutative pillow algebra, the quantum polynomial algebras, Hopf algebra domains, families of Ore extensions, some 3-dimensional skew polynomial algebras, diffusion algebras in three generators, and noncommutative coordinate algebras of deformations of examples of classical orbifolds. Recently, in a series of papers \cite{Rubiano2026} - \cite{RubianoReyes2025BAG}, the authors considered the question on the differential smoothness of different types of algebra appearing in noncommutative algebraic geometry. In this paper, the objects of interest are the {\em 3-dimensional skew polynomial rings} introduced by Bell and Smith \cite{BellSmith1990}, which have been studied from the point of view of ring theory and noncommutative algebraic geometry (e.g. \cite{ChaconReyes2025, Fajardoetal2020, NinoReyes2025, RedmanPhDThesis1996, Redman1999, ReyesRodriguez2021, ReyesSuarez20173D, Rosenberg1995}, and references therein). 

The article is organized as follows. In Section \ref{PreliminariesDifferentialsmoothnessofbi-quadraticalgebras}, we consider the preliminaries on differential smoothness of algebras and 3-dimensional skew polynomial rings, in order to set up notation and render this paper self-contained. Section \ref{Differentialandintegralcalculusbi-quadraticalgebras} contains the original results of the paper: Theorem \ref{Firsttheoremsmoothness3-dimensionalalgebras} formulates sufficient conditions to guarantee that a 3-dimensional skew polynomial ring is differentially smooth, and Theorem \ref{Secondtheoremsmoothness3-dimensionalalgebras} establishes sufficient conditions to assert that there are no one-dimensional connected integrable calculi over these rings. Section \ref{Futurework} presents some ideas for possible future work.

Throughout the paper, $\mathbb{N}$ denotes the set of natural numbers including zero. The word ring means an associative ring with identity not necessarily commutative. All vector spaces and algebras (always associative and with unit) are over a fixed field $\Bbbk$. $\Bbbk^{\times}$ denotes the non-zero elements of $\Bbbk$. As usual, the symbols $\mathbb{R}$ and $\mathbb{C}$ denote the fields of real and complex numbers, respectively. As usual, $|\square|$ means the number of elements of the set $\square$. ${\rm Aut}(R)$ is the set of automorphisms of the ring $R$.

\section{Definitions and Preliminaries}\label{PreliminariesDifferentialsmoothnessofbi-quadraticalgebras}

\subsection{Differential smoothness of algebras}\label{DefinitionsandpreliminariesDSA}

% Following Giachetta et al. \cite[Section 1.6]{Giachettaetal2005}, a {\em graded algebra} $\Omega$ over $\Bbbk$ is defined as a direct sum $\Omega = \bigoplus\limits_{n\in \mathbb{N}} \Omega^{n}$ of $\Bbbk$-modules, provided with an associative multiplication law $a \wedge b,\ a, b \in \Omega$, such that $a \wedge b\in \Omega^{|a| + |b|}$, where $|a|$ denotes the degree of an element $a \in \Omega^{|n|}$. Notice that $\Omega^0$ is a (noncommutative) $\Bbbk$-algebra $A$, while $\Omega^{n > 0}$ are $A$-bimodules and $\Omega$ is an $(A-A)-$algebra. A graded algebra $\Omega$ is said to be {\em graded commutative} if $a \wedge b = (-1)^{|a||b|} b \wedge a$ for elements $a, b \in \Omega$.

% A graded algebra $\Omega = \bigoplus_{n \in \mathbb{N}} \Omega^n$, where $\Omega^{0} \cong A$, is said to be a {\em differential calculus} over $A$ if it is a cochain complex of $\Bbbk$-modules (also known as {\em de Rham complex} of the differential graded algebra $(\Omega, d)$)
% \[
% 0 \to \Bbbk \to A \xrightarrow{d} \Omega^{1}  \xrightarrow{d} \dotsb \xrightarrow{d} \Omega^{n} \xrightarrow{d} \dotsb, 
% \]

% with respect to a coboundary operator $d$ which obeys the {\em graded Leibniz's rule} $d(a \wedge b) = d(a) \wedge b + (-1)^{|{\rm deg}(a)|} a \wedge d(b)$, for all pairs of homogeneous elements $a, b \in \Omega$. In particular, $d: A \to \Omega^{1}$ is a $\Omega^{1}$-valued derivation of a $\Bbbk$-algebra $A$.

We follow Brzezi\'nski and Sitarz's ideas on differential smoothness carried out in \cite[Section 2]{BrzezinskiSitarz2017} (see also \cite{Brzezinski2008, Brzezinski2014, BrzezinskiElKaoutitLomp2010}).

\begin{definition}[{\cite[Section 2.1]{BrzezinskiSitarz2017}}]
\begin{enumerate}
    \item [\rm (i)] A {\em differential graded algebra} is a non-negatively graded algebra $\Omega$ with the product denoted by $\wedge$ together with a degree-one linear map $d:\Omega^{\bullet} \to \Omega^{\bullet +1}$ that satisfies the graded Leibniz's rule and is such that $d \circ d = 0$. 
    
    \item [\rm (ii)] A differential graded algebra $(\Omega, d)$ is a {\em calculus over an algebra} $A$ if 
    $$
    \Omega^0 A = A \quad {\rm and} \quad \Omega^n A = A\ dA \wedge dA \wedge \dotsb \wedge dA \quad (dA \, \, {\rm appears} \, n-{\rm times}), 
    $$
    
    for all $n\in \mathbb{N}$ (this last is called the {\em density condition}). We write $(\Omega A, d)$ with $\Omega A = \bigoplus_{n\in \mathbb{N}} \Omega^{n}A$. By using the Leibniz's rule, it follows that 
    $$
    \Omega^n A = dA \wedge dA \wedge \dotsb \wedge dA\ A.
    $$
    
    A differential calculus $\Omega A$ is called {\em connected} if ${\rm ker}(d\mid_{\Omega^0 A}) = \Bbbk$.
    
    \item [\rm (iii)] A calculus $(\Omega A, d)$ is said to have {\em dimension} $n$ if $\Omega^n A\neq 0$ and $\Omega^m A = 0$ for all $m > n$. An $n$-dimensional calculus $\Omega A$ {\em admits a volume form} if $\Omega^n A$ is isomorphic to $A$ as a left and right $A$-module. 
\end{enumerate}
\end{definition}

The existence of a right $A$-module isomorphism means that there is a free generator, say $\omega$, of $\Omega^n A$ (as a right $A$-module), i.e. $\omega \in \Omega^n A$, such that all elements of $\Omega^n A$ can be uniquely expressed as $\omega a$ with $a \in A$. If $\omega$ is also a free generator of $\Omega^n A$ as a left $A$-module, it is said to be a {\em volume form} on $\Omega A$.

The right $A$-module isomorphism $\Omega^n A \to A$ corresponding to a volume form $\omega$ is denoted by $\pi_{\omega}$, i.e.
\begin{equation}\label{BrzezinskiSitarz2017(2.1)}
\pi_{\omega} (\omega a) = a, \quad {\rm for\ all}\ a\in A.
\end{equation}

By using that $\Omega^n A$ is also isomorphic to $A$ as a left $A$-module, any free generator $\omega $ induces an algebra endomorphism $\nu_{\omega}$ of $A$ by the formula
\begin{equation}\label{BrzezinskiSitarz2017(2.2)}
    a \omega = \omega \nu_{\omega} (a).
\end{equation}

Note that if $\omega$ is a volume form, then $\nu_{\omega}$ is an algebra automorphism.

Next, we recall the key ingredients of the {\em integral calculus} on $A$ as dual to its differential calculus. 

Let $(\Omega A, d)$ be a differential calculus on $A$. The space of $n$-forms $\Omega^n A$ is an $A$-bimodule. Consider $\mathcal{I}_{n}A$ the right dual of $\Omega^{n}A$, the space of all right $A$-linear maps $\Omega^{n}A\rightarrow A$, that is, $\mathcal{I}_{n}A := {\rm Hom}_{A}(\Omega^{n}(A),A)$. Notice that each of the $\mathcal{I}_{n}A$ is an $A$-bimodule with the actions given by
\begin{align*}
    (a\cdot\phi\cdot b)(\omega)=a\phi(b\omega),\quad {\rm for\ all}\ \phi \in \mathcal{I}_{n}A,\ \omega \in \Omega^{n}A\ {\rm and}\ a,b \in A.
\end{align*}

The direct sum of all the $\mathcal{I}_{n}A$, that is, $\mathcal{I}A = \bigoplus_{n} \mathcal{I}_n A$, is a right $\Omega A$-module with action given by
\begin{align}\label{BrzezinskiSitarz2017(2.3)}
    (\phi\cdot\omega)(\omega')=\phi(\omega\wedge\omega'),\quad {\rm for\ all}\ \phi\in\mathcal{I}_{n + m}A, \ \omega\in \Omega^{n}A \ {\rm and} \ \omega' \in \Omega^{m}A.
\end{align}

\begin{definition}[{\cite[Definition 2.1]{Brzezinski2008}}]
A {\em divergence} (also called {\em hom-connection}) on $A$ is a linear map $\nabla: \mathcal{I}_1 A \to A$ such that
\begin{equation}\label{BrzezinskiSitarz2017(2.4)}
    \nabla(\phi \cdot a) = \nabla(\phi) a + \phi(da), \quad {\rm for\ all}\ \phi \in \mathcal{I}_1 A \ {\rm and} \ a \in A.
\end{equation}  
\end{definition}

Note that a divergence can be extended to the whole of $\mathcal{I}A$, 
\[
\nabla_n: \mathcal{I}_{n+1} A \to \mathcal{I}_{n} A,
\]

by considering
\begin{equation}\label{BrzezinskiSitarz2017(2.5)}
\nabla_n(\phi)(\omega) = \nabla(\phi \cdot \omega) + (-1)^{n+1} \phi(d \omega), \quad {\rm for\ all}\ \phi \in \mathcal{I}_{n+1}(A)\ {\rm and} \ \omega \in \Omega^n A.
\end{equation}

By putting together (\ref{BrzezinskiSitarz2017(2.4)}) and (\ref{BrzezinskiSitarz2017(2.5)}), we get the Leibniz's rule 
\begin{equation}
    \nabla_n(\phi \cdot \omega) = \nabla_{m + n}(\phi) \cdot \omega + (-1)^{m + n} \phi \cdot d\omega,
\end{equation}

for all elements $\phi \in \mathcal{I}_{m + n + 1} A$ and $\omega \in \Omega^m A$ \cite[Lemma 3.2]{Brzezinski2008}. In the case $n = 0$, if ${\rm Hom}_A(A, M)$ is canonically identified with $M$, then $\nabla_0$ reduces to the classical Leibniz's rule.

\begin{definition}[{\cite[Definition 3.4]{Brzezinski2008}}]
The right $A$-module map 
\[
F = \nabla_0 \circ \nabla_1: {\rm Hom}_A(\Omega^{2} A, M) \to M
\] 

is called a {\em curvature} of a hom-connection $(M, \nabla_0)$. $(M, \nabla_0)$ is said to be {\em flat} if its curvature is the zero map, that is, if $\nabla \circ \nabla_1 = 0$. This condition implies that $\nabla_n \circ \nabla_{n+1} = 0$ for all $n\in \mathbb{N}$.
\end{definition}

$\mathcal{I} A$ together with the $\nabla_n$ form a chain complex called the {\em complex of integral forms} over $A$. The cokernel map of $\nabla$, that is, $\Lambda: A \to {\rm Coker} \nabla = A / {\rm Im} \nabla$ is said to be the {\em integral on $A$ associated to} $\mathcal{I}A$.

Given a left $A$-module $X$ with action $a\cdot x$ for all $a\in A,\ x \in X$, and an algebra automorphism $\nu$ of $A$, the notation $^{\nu}X$ stands for $X$ with the $A$-module structure twisted by $\nu$, i.e. with the $A$-action $a\otimes x \mapsto \nu(a)\cdot x $.

The following definition of an \textit{integrable differential calculus} seeks to portray a version of Hodge star isomorphisms between the complex of differential forms of a differentiable manifold and a complex of dual modules of it \cite[p. 112]{Brzezinski2015}. %An integrable calculus can be characterized by the existence of a bimodule complex $(\mathcal{I}_{\bullet} A, \nabla)$ known as the {\em complex of integral forms} isomorphic to the de Rham complex $(\Omega A, d)$.

\begin{definition}[{\cite[Definition 2.1]{BrzezinskiSitarz2017}}]
An $n$-dimensional differential calculus $(\Omega A, d)$ is said to be {\em integrable} if $(\Omega A, d)$ admits a complex of integral forms $(\mathcal{I}A, \nabla)$, for which there exist an algebra automorphism $\nu$ of $A$ and $A$-bimodule isomorphisms \linebreak $\Theta_k: \Omega^{k} A \to ^{\nu} \mathcal{I}_{n-k}A$, $k = 0, \dotsc, n$, rendering commmutative the following diagram:
\[
\begin{tikzcd}
A \arrow{r}{d} \arrow{d}{\Theta_0} & \Omega^{1} A \arrow{d}{\Theta_1} \arrow{r}{d} & \Omega^2 A  \arrow{d}{\Theta_2} \arrow{r}{d} & \dotsb \arrow{r}{d} & \Omega^{n-1} A \arrow{d}{\Theta_{n-1}} \arrow{r}{d} & \Omega^n A  \arrow{d}{\Theta_n} \\ ^{\nu} \mathcal{I}_n A \arrow[swap]{r}{\nabla_{n-1}} & ^{\nu} \mathcal{I}_{n-1} A \arrow[swap]{r}{\nabla_{n-2}} & ^{\nu} \mathcal{I}_{n-2} A \arrow[swap]{r}{\nabla_{n-3}} & \dotsb \arrow[swap]{r}{\nabla_{1}} & ^{\nu} \mathcal{I}_{1} A \arrow[swap]{r}{\nabla} & ^{\nu} A
\end{tikzcd}
\]

The $n$-form $\omega:= \Theta_n^{-1}(1)\in \Omega^n A$ is called an {\em integrating volume form}. 
\end{definition}

The algebra of complex matrices $M_n(\mathbb{C})$ with the $n$-dimensional calculus generated by derivations presented by Dubois-Violette et al. \cite{DuboisViolette1988, DuboisVioletteKernerMadore1990}, the quantum group $SU_q(2)$ with the three-dimensional left covariant calculus developed by Woronowicz \cite{Woronowicz1987} and the quantum standard sphere with the restriction of the above calculus, are examples of algebras admitting integrable calculi. 

The following proposition shows that the integrability of a differential calculus can be defined without explicit reference to integral forms. This allows us to guarantee the integrability by considering the existence of finitely generator elements that allow to determine left and right components of any homogeneous element of $\Omega(A)$.

\begin{proposition}[{\cite[Theorem 2.2]{BrzezinskiSitarz2017}}]\label{integrableequiva} 
Let $(\Omega A, d)$ be an $n$-dimensional differential calculus over an algebra $A$. The following assertions are equivalent:
\begin{enumerate}
    \item [\rm (1)] $(\Omega A, d)$ is an integrable differential calculus.
    
    \item [\rm (2)] There exists an algebra automorphism $\nu$ of $A$ and $A$-bimodule isomorphisms 
    $$\Theta_k : \Omega^k A \rightarrow \ ^{\nu}\mathcal{I}_{n-k}A, \quad k =0, \ldots, n,
    $$
    
    such that for all $\omega'\in \Omega^k A$ and $\omega''\in \Omega^mA$, we have that 
    \begin{align*}
        \Theta_{k+m}(\omega'\wedge\omega'')=(-1)^{(n-1)m}\Theta_k(\omega')\cdot\omega''.
    \end{align*}
    
    \item [\rm (3)] There exists an algebra automorphism $\nu$ of $A$ and an $A$-bimodule map $\vartheta:\Omega^nA\rightarrow\ ^{\nu}A$ such that all left multiplication maps
    \begin{align*}
    \ell_{\vartheta}^{k}:\Omega^k A &\ \rightarrow \mathcal{I}_{n-k}A, \\
    \omega' &\ \mapsto \vartheta\cdot\omega', \quad k = 0, 1, \dotsc, n,
    \end{align*}
    where the actions $\cdot$ are defined by {\rm (}\ref{BrzezinskiSitarz2017(2.3)}{\rm )}, are bijective.
    
    \item [\rm (4)] $(\Omega A, d)$ has a volume form $\omega$ such that all left multiplication maps
    \begin{align*}
        \ell_{\pi_{\omega}}^{k}:\Omega^k A &\ \rightarrow \mathcal{I}_{n-k}A, \\
        \omega' &\ \mapsto \pi_{\omega} \cdot \omega', \quad k=0,1, \dotsc, n-1,
    \end{align*}
    
    where $\pi_{\omega}$ is defined by {\rm (}\ref{BrzezinskiSitarz2017(2.1)}{\rm )}, are bijective.
\end{enumerate}
\end{proposition}

A volume form $\omega\in \Omega^nA$ is an {\em integrating form} if and only if it satisfies condition $(4)$ of Proposition \ref{integrableequiva} \cite[Remark 2.3]{BrzezinskiSitarz2017}.

The most interesting cases of differential calculi are those where $\Omega^k A$ are finitely generated and projective right or left (or both) $A$-modules (for more details, see \cite{Brzezinski2011}).

\begin{proposition}\label{BrzezinskiSitarz2017Lemmas2.6and2.7}
\begin{enumerate}
\item [\rm (1)] \cite[Lemma 2.6]{BrzezinskiSitarz2017} Consider $(\Omega A, d)$ an integrable and $n$-dimensional calculus over $A$ with integrating form $\omega$. Then $\Omega^{k} A$ is a finitely generated projective right $A$-module if there exist a finite number of forms $\omega_i \in \Omega^{k} A$ and $\overline{\omega}_i \in \Omega^{n-k} A$ such that for all $\omega' \in \Omega^{k} A$, we have that 
\begin{equation*}
\omega' = \sum_{i} \omega_i \pi_{\omega} (\overline{\omega}_i \wedge \omega').
\end{equation*}

\item [\rm (2)] \cite[Lemma 2.7]{BrzezinskiSitarz2017} Let $(\Omega A, d)$ be an $n$-dimensional calculus over $A$ admitting a volume form $\omega$. Assume that for all $k = 1, \ldots, n-1$, there exists a finite number of forms $\omega_{i}^{k},\overline{\omega}_{i}^{k} \in \Omega^{k}(A)$ such that for all $\omega'\in \Omega^kA$, we have that
\begin{equation*}
\omega'=\displaystyle\sum_i\omega_{i}^{k}\pi_\omega(\overline{\omega}_{i}^{n-k}\wedge\omega')=\displaystyle\sum_i\nu_{\omega}^{-1}(\pi_\omega(\omega'\wedge\omega_{i}^{n-k}))\overline{\omega}_{i}^{k},
\end{equation*}

where $\pi_{\omega}$ and $\nu_{\omega}$ are defined by {\rm (}\ref{BrzezinskiSitarz2017(2.1)}{\rm )} and {\rm (}\ref{BrzezinskiSitarz2017(2.2)}{\rm )}, respectively. Then $\omega$ is an integral form and all the $\Omega^{k}A$ are finitely generated and projective as left and right $A$-modules.
\end{enumerate}
\end{proposition}

Brzezi\'nski and Sitarz \cite[p. 421]{BrzezinskiSitarz2017} point out that, in order to relate the dimension of an integrable calculus $(\Omega A,d)$ with the ``size'' of the underlying affine algebra $A$, a suitable notion of dimension is needed: its Gelfand-Kirillov dimension $\text{GKdim}(A)$ introduced by Gelfand and Kirillov \cite{GelfandKirillov1966, GelfandKirillov1966b} (in some sense, this dimensions measures the deviation of the algebra $A$ from finite dimensionality). For more details about this dimension, see the excellent treatment by Krause and Lenagan \cite{KrauseLenagan2000}.

We arrive to the key notion of the paper.

\begin{definition}[{\cite[Definition 2.4]{BrzezinskiSitarz2017}}]\label{BrzezinskiSitarz2017Definition2.4}
An affine algebra $A$ with integer Gelfand-Kirillov dimension $n$ is said to be {\em differentially smooth} if it admits an $n$-dimensional connected integrable differential calculus $(\Omega A, d)$.
\end{definition}

From Definition \ref{BrzezinskiSitarz2017Definition2.4} it follows that a differentially smooth algebra comes equipped with a well-behaved differential structure and with the precise concept of integration \cite[p. 2414]{BrzezinskiLomp2018}.

As we said in the Introduction, several examples of noncommutative algebras have been proved to be differentially smooth (e.g. \cite{Brzezinski2015, BrzezinskiElKaoutitLomp2010, BrzezinskiLomp2018, BrzezinskiSitarz2017, Karacuha2015, KaracuhaLomp2014, ReyesSarmiento2022}). Of course, there are examples of algebras that are not differentially smooth. Consider the commutative algebra $A = \mathbb{C}[x, y] / \langle xy \rangle$. A proof by contradiction shows that for this algebra there are no one-dimensional connected integrable calculi over $A$, so it cannot be differentially smooth \cite[Example 2.5]{BrzezinskiSitarz2017}.

\subsection{3-dimensional skew polynomial algebras}\label{3-dimensionaldefinitioncharacterization}
Bell and Smith \cite{BellSmith1990} (see also \cite[Definition C4.3]{Rosenberg1995}) defined the {\em 3-dimensional skew polynomial rings} as those $\Bbbk$-algebras generated by the indeterminates $x, y, z$ restricted to relations 
$$
yz-\alpha zy=\lambda,\quad zx-\beta xz=\mu \quad {\rm and}\ xy-\gamma yx = \nu,
$$ 

such that
\begin{enumerate}
\item [\rm (i)] $\lambda, \mu, \nu\in \Bbbk+\Bbbk x+\Bbbk y+\Bbbk z$, and $\alpha, \beta, \gamma \in \Bbbk\ \backslash\ \{0\}$, and

\item [\rm (ii)] standard monomials $\left\{x^iy^jz^l\mid i,j,l\ge 0\right\}$ are a $\Bbbk$-basis of the algebra $A$.
\end{enumerate}

The classification of these skew polynomial rings is presented in the following proposition.

\begin{proposition}[{\cite[Theorem C.4.3.1]{Rosenberg1995}}]\label{3-dimensionalClassification}
If $A$ is a 3-dimensional skew polynomial ring, then $A$ is one of the following algebras:
\begin{enumerate}
\item [\rm (1)] if\ $|\{\alpha, \beta, \gamma\}|=3$, then $A$ is given by the relations 
$$
yz - \alpha zy = 0,\ zx - \beta xz = 0\ {\rm and} \ xy - \gamma yx = 0.
$$

\item [\rm (2)] if $|\{\alpha, \beta, \gamma\}|=2$ and $\beta\neq \alpha =\gamma =1$, then $A$ is one of the following algebras:
\begin{enumerate}
\item [\rm (i)] $yz-zy=z,\ zx-\beta xz=y$ and $xy - yx = x$ {\rm (}if $\beta = -1$, then we get the Dispin algebra{\rm )}.

\item [\rm (ii)] $yz-zy=z,\ zx-\beta xz=b$ and $xy-yx=x${\rm ;}

\item [\rm (iii)] $yz-zy=0,\ zx-\beta xz=y$ and $xy - yx = 0${\rm ;}

\item [\rm (iv)] $yz-zy=0,\ zx-\beta xz=b$ and $xy - yx = 0${\rm ;}

\item [\rm (v)] $yz-zy=az,\ zx-\beta xz = 0$ and $xy - yx = x${\rm ;}

\item [\rm (vi)] $yz-zy=z,\ zx-\beta xz=0$ and $xy - yx = 0$,
\end{enumerate}

where $a, b$ are any elements of\ $\Bbbk$. All non-zero values of $b$
give isomorphic algebras.

\item [\rm (3)] If $|\{\alpha, \beta, \gamma\}|=2$ and $\beta\neq \alpha=\gamma\neq 1$, then $A$ is one of the following algebras:
\begin{enumerate}
\item [\rm (i)] $yz-\alpha zy=0,\ zx-\beta xz = y+b$ and $xy - \alpha yx=0${\rm ;}

\item [\rm (ii)] $yz-\alpha zy=0,\ zx - \beta xz=b$ and $xy - \alpha yx = 0$.
\end{enumerate}

In this case, $b$ is an arbitrary element of $\Bbbk$. Again, any
non-zero values of $b$ give isomorphic algebras.

\item [\rm (4)] If $\alpha=\beta=\gamma\neq 1$, then $A$ is the algebra defined by the relations
\begin{align*}
yz - \alpha zy = &\ a_1x+b_1, \\
zx - \alpha xz = &\ a_2y+b_2, \ {\rm and} \\
xy - \alpha yx = &\ a_3z+b_3.
\end{align*}

If $a_i = 0$ for all $i$, then all non-zero values of $b_i$ give isomorphic
algebras.

\item [\rm (5)] If $\alpha=\beta=\gamma=1$, then $A$ is isomorphic to one of the following algebras:
\begin{enumerate}
\item [\rm (i)] $yz-zy=x,\ zx-xz=y$ and $xy - yx = z${\rm ;}

\item [\rm (ii)] $yz-zy=0,\ zx-xz=0$ and $xy - yx = z${\rm ;}

\item [\rm (iii)] $yz-zy=0,\ zx-xz=0$ and $xy - yx = b${\rm ;}

\item [\rm (iv)] $yz-zy=-y,\ zx-xz=x+y$ and $xy-yx=0${\rm ;}

\item [\rm (v)] $yz-zy=az,\ zx-xz=z$ and $xy - yx = 0${\rm ;}
\end{enumerate}

Parameters $a,b\in \Bbbk$ are arbitrary,  and all non-zero values of
$b$ generate isomorphic algebras.
\end{enumerate}
\end{proposition}

\begin{remark}\label{Rosenberg}
As wee saw above, Rosenberg \cite[Theorem C4.3.1, case (5)(v), p. 101]{Rosenberg1995} presented the relations given by
\[
yz-zy=az,\ zx-xz=x \quad {\rm and} \quad xy-yx=0.
\]

However, there is a typo in the second relation as one can check using Bergman's diamond lemma \cite{Bergman1978} to guarantee the PBW basis of the algebra (see Remark \ref{PBWbasis3D} for more details). As a matter of fact, he presented this algebra as the Ore extension \cite{Ore1933} given by
\[
A = \Bbbk[x,y][z;\varphi],
\]

where $\varphi\in\mathrm{Aut}(\Bbbk[x,y])$ is given by $\varphi(x)=x+1$ and $\varphi(y)=y-a$ \cite[p. 108]{Rosenberg1995}. By the defining rule $za=\varphi(a)z$ for all $a\in \Bbbk[x,y]$, we obtain that
\begin{align*}
zx=\varphi(x)z=(x+1)z &\quad {\rm whence} \quad  zx-xz=z, \\
zy=\varphi(y)z=(y-a)z &\quad {\rm i.e.} \quad yz-zy=az.
\end{align*}

Therefore, the correct defining relations of case (5)(v) are given by 
\[
yz - zy = az,\ zx - xz = z,\quad {\rm and}\quad xy - yx = 0.
\]
\end{remark}

Let us see some remarkable examples of 3-dimensional skew polynomial rings appearing in Bavula \cite{Bavula2023}.

\begin{example}
\begin{enumerate}
\item [\rm (i)] The commutative polynomial ring $\Bbbk[x, y, z]$ in three indeterminates.

\item [\rm (ii)] The universal enveloping algebra $U(\mathfrak{sl}(2))$ of the Lie algebra $\mathfrak{sl}(2)$.

\item [\rm (iii)] The quotient algebra $U(\mathcal{N}) / \langle c - 1\rangle$ given by the relations
\begin{align*}
xy - yx = &\ z, \\
xz - zx = &\ 0 \quad {\rm and} \\
yz - zy = &\ 1.
\end{align*}

As it can be seen, $U(\mathcal{N}) / \langle c - 1\rangle$ is a tensor product $A_1 \otimes_{\Bbbk} \, \Bbbk[x']$ of its subalgebras: the Weyl algebra $A_1 = \Bbbk \{y, z\} / \langle yz - zy - 1\rangle$ and the polynomial algebra $\Bbbk[x']$ where $x' = x + \frac{1}{2}z^2$.

\item [\rm (iv)] 
$$
U(\mathfrak{n}_2 \times \Bbbk z) \cong \Bbbk \{x, y, z\} / \langle xy - yx - y\rangle, 
$$ 

where $z$ is a central element.

\item [\rm (v)] 
$$
U(\mathcal{M})/\langle c-1\rangle \cong \Bbbk\{x,y,z\} / \langle xy - yx = y, \; xz - zx = 1,\; yz - zy = 0 \rangle.
$$ 

Note that $U(\mathcal{M})/\langle c-1\rangle$ is a skew polynomial algebra $A_1[y;\sigma]$ where
$$
A_1 = \Bbbk\{x, z\} / \langle xz - zx - 1\rangle
$$

is the Weyl algebra and $\sigma$ is an automorphism of $A_1$ given by the rule $\sigma(x) = x + 1$ and $\sigma(z) = z$.

\item [\rm (vi)] Zhedanov \cite[Section 1]{Zhedanov1991} introduced the {\em Askey-Wilson algebras} $AW(3)$ as the algebra generated by three operators $K_0, K_1$, and $K_2$, that satisfy the commutation relations 
\begin{align*}
[K_0, K_1]_{\omega} = &\, K_2, \\    
[K_2, K_0]_{\omega} = &\, BK_0 + C_1K_1 + D_1 \quad {\rm and} \\
[K_1, K_2]_{\omega} = &\ BK_1 + C_0K_0 + D_0
\end{align*}

where $B, C_0, C_1, D_0$, and $D_1$ are the structure constants of the algebra, which Zhedanov assumes are real, and the $q$-commutator $[ - , -]_{\omega}$ is given by 
$$
[\square, \triangle]_{\omega}:= e^{\omega}\square \triangle - e^{- \omega}\triangle \square, \quad {\rm where} \, \, \omega\in \mathbb{R}^{\times}.
$$ 

Notice that in the limit $\omega \to 0$, the algebra AW(3) becomes an ordinary Lie algebra with three generators ($D_0$ and $D_1$ are included among the structure constants of the algebra in order to take into account algebras of Heisenberg-Weyl type). The relations defining the algebra can be written as 
\begin{align*}
    e^{\omega}K_0K_1 - e^{-\omega}K_1K_0 = &\ K_2,\\
    e^{\omega} K_2K_0 - e^{-\omega}K_0 K_2 = &\ BK_0 + C_1K_1 + D_1 \quad {\rm and} \\
    e^{\omega}K_1K_2 - e^{-\omega}K_2K_1 = &\ BK_1 + C_0K_0 + D_0.
\end{align*}

\item [\rm (vii)] The {\em Dispin algebra} $U({\mathfrak{osp}}(1,2))$ is generated by the indeterminates $x, y, z$ over the field $\Bbbk$ subject to the relations
\begin{align*}
yz - zy = &\, z, \\
zx + xz = &\, y \quad {\rm and} \\
xy - yx = &\, x. 
\end{align*}

\item [\rm (viii)] The {\em Woronowicz algebra} $W_v(\mathfrak{sl}(2,\Bbbk))$ was introduced by Woronowicz \cite{Woronowicz1987} and it is generated by $x, y, z$ with defining relations given by
\begin{align*}
xz - \nu^{4}zx = &\, (1 + \nu^2)x,  \\
xy - \nu^{2}yx = &\, \nu z \quad {\rm and} \\
zy - \nu^{4}yz = &\, (1 + \nu^{2})y
\end{align*}

\item [\rm (ix)] Smith \cite{Smith1990} studied similar and different algebras to the universal enveloping algebra from $U(\mathfrak{sl}(2, \mathbb{C}))$, the universal enveloping algebra of $\mathfrak{sl}(2, \mathbb{C})$. Fix $f(h)\in \mathbb{C}[h]$ and define the $\mathbb{C}$-algebra $A$ subject to the relations
\begin{align*}
hx - xh = &\ x, \\
hy - yh = &\ -y \quad {\rm and} \\
xy - yx = &\ f(H).
\end{align*}

If ${\rm deg}(f) \le 1$, then $A$ is a factor ring of an enveloping algebra. In particular, if $xy - yx = \alpha h + \beta$ with $\alpha, \beta \in \mathbb{C}$, then we have the following possibilities:
\begin{itemize}
    \item if $\alpha = 0$ and $\beta \neq 0$, then $A \cong \mathbb{C}[t, \partial] \otimes_{\mathbb{C}} \mathbb{C}[s]$, where $t$ and $s$ are commuting indeterminates and $\partial = d/dt$.

    \item If $\alpha = 0$ and $\beta = 0$, then $A\cong U(\mathfrak{h})$ where $\mathfrak{h}$ is a 3-dimensional solvable Lie algebra.

    \item If $\alpha \neq 0$, then $A\cong U(\mathfrak{sl}(2))$.
\end{itemize}

\item [\rm (x)] Following Havli\v{c}ek et al. \cite[p. 79]{HavlicekKlimykPosta2000} (see also \cite{Fairlie1990, Odesskii1986}), the $ \mathbb{C}$-algebra $U_q'(\mathfrak{so}_3)$ is generated by the indeterminates $I_1, I_2$, and $I_3$, subject to the relations given by
\begin{align*}
I_2I_1 - qI_1I_2 = &\, -q^{\frac{1}{2}}I_3, \\ 
I_3I_1 - q^{-1}I_1I_3 = &\, q^{-\frac{1}{2}}I_2 \quad {\rm and} \\ 
I_3I_2 - qI_2I_3 = &\, -q^{\frac{1}{2}}I_1,
\end{align*}

where $q$ is a non-zero element of $\mathbb{C}$. It is straightforward to show that $U_q'(\mathfrak{so}_3)$ cannot be expressed as an iterated Ore extension of $\mathbb{C}$.  
\end{enumerate}
\end{example}

\begin{remark}\label{PBWbasis3D}
From Section \ref{3-dimensionaldefinitioncharacterization}, consider the notation given by 
\begin{align}
yz-\alpha zy & = a_{\lambda}x+b_{\lambda}y+c_{\lambda}z+d_{\lambda}, \notag \\
zx-\beta xz & = a_{\mu}x+b_{\mu}y+c_{\mu}z+d_{\mu}  \quad {\rm and}\ \label{rel3skew} \\
xy-\gamma yx & = a_{\nu}x+b_{\nu}y+c_{\nu}z+d_{\nu}, \notag
\end{align} 

with the elements $a$'s, $b$'s, $c$'s, $d$'s belonging to $\Bbbk$ and $\alpha, \beta, \gamma \in \Bbbk\setminus\{0\}$.

Fix $x\prec y\prec z$. Using Bergman's diamond lemma, necessary and sufficient conditions are established to guarantee that the standard monomials $\left\{x^iy^jz^l\mid i,j,l\ge 0\right\}$ are a $\Bbbk$-basis of 3-dimensional skew polynomial rings. Taking into account the notation above, these relations are obtained as follows (c.f. \cite[Section 5]{ReyesSuarez20173D}).

Since $\alpha,\beta,\gamma\neq 0$, the defining relations yield the rewriting rules
\begin{align*}
f_{21} = &\ \gamma^{-1}xy-\gamma^{-1}a_{\nu}x-\gamma^{-1}b_{\nu}y-\gamma^{-1}c_{\nu}z-\gamma^{-1}d_{\nu},\\
f_{32} = &\ \alpha^{-1}yz-\alpha^{-1}a_{\lambda}x-\alpha^{-1}b_{\lambda}y-\alpha^{-1}c_{\lambda}z-\alpha^{-1}d_{\lambda},\\
f_{31} = &\ \beta xz+a_{\mu}x+b_{\mu}y+c_{\mu}z+d_{\mu},
\end{align*}

it is enough to check that both reductions of $zyx$ coincide, namely
\begin{align}\label{Diamond3DBergman}
   z f_{21}=f_{32}x. 
\end{align}

Note that
\begin{align*}
z f_{21}
&=z\bigl(\gamma^{-1}xy-\gamma^{-1}a_{\nu}x-\gamma^{-1}b_{\nu}y-\gamma^{-1}c_{\nu}z-\gamma^{-1}d_{\nu}\bigr)\\
&=\gamma^{-1}zxy-\gamma^{-1}a_{\nu}zx-\gamma^{-1}b_{\nu}zy-\gamma^{-1}c_{\nu}z^2-\gamma^{-1}d_{\nu}z.
\end{align*}

Using the rules for $zx$ and $zy$, we obtain the equalities 
\begin{align*}
zxy
& = (zx)y\\
& = (\beta xz+a_{\mu}x+b_{\mu}y+c_{\mu}z+d_{\mu})y\\
& = \beta xzy+a_{\mu}xy+b_{\mu}y^2+c_{\mu}zy+d_{\mu}y\\
& = \beta x\bigl(\alpha^{-1}yz-\alpha^{-1}a_{\lambda}x-\alpha^{-1}b_{\lambda}y-\alpha^{-1}c_{\lambda}z-\alpha^{-1}d_{\lambda}\bigr)\\
&  \quad + a_{\mu}xy+b_{\mu}y^2+c_{\mu}\bigl(\alpha^{-1}yz-\alpha^{-1}a_{\lambda}x-\alpha^{-1}b_{\lambda}y-\alpha^{-1}c_{\lambda}z-\alpha^{-1}d_{\lambda}\bigr)+d_{\mu}y\\
& = \frac{\beta}{\alpha}xyz-\frac{\beta a_{\lambda}}{\alpha}x^2-\frac{\beta b_{\lambda}}{\alpha}xy-\frac{\beta c_{\lambda}}{\alpha}xz-\frac{\beta d_{\lambda}}{\alpha}x\\
&  \quad + a_{\mu}xy+b_{\mu}y^2+\frac{c_{\mu}}{\alpha}yz-\frac{a_{\lambda}c_{\mu}}{\alpha}x-\frac{b_{\lambda}c_{\mu}}{\alpha}y-\frac{c_{\lambda}c_{\mu}}{\alpha}z-\frac{c_{\mu}d_{\lambda}}{\alpha}+d_{\mu}y.
\end{align*}

Also,
\begin{align*}
-a_{\nu}zx
&=-a_{\nu}\bigl(\beta xz+a_{\mu}x+b_{\mu}y+c_{\mu}z+d_{\mu}\bigr)\\
&=-\beta a_{\nu}xz-a_{\mu}a_{\nu}x-a_{\nu}b_{\mu}y-a_{\nu}c_{\mu}z-a_{\nu}d_{\mu},
\end{align*}

and
\begin{align*}
-b_{\nu}zy
&=-b_{\nu}\bigl(\alpha^{-1}yz-\alpha^{-1}a_{\lambda}x-\alpha^{-1}b_{\lambda}y-\alpha^{-1}c_{\lambda}z-\alpha^{-1}d_{\lambda}\bigr)\\
&=-\frac{b_{\nu}}{\alpha}yz+\frac{a_{\lambda}b_{\nu}}{\alpha}x+\frac{b_{\lambda}b_{\nu}}{\alpha}y+\frac{b_{\nu}c_{\lambda}}{\alpha}z+\frac{b_{\nu}d_{\lambda}}{\alpha}.
\end{align*}

In this way, 
\begin{align*}
z f_{21}
& = \frac{\beta}{\alpha\gamma}xyz
-\frac{\beta a_{\lambda}}{\alpha\gamma}x^2
+\frac{\alpha a_{\mu}-\beta b_{\lambda}}{\alpha\gamma}xy
-\frac{\beta(\alpha a_{\nu}+c_{\lambda})}{\alpha\gamma}xz
+\frac{b_{\mu}}{\gamma}y^2 + \frac{-b_{\nu}+c_{\mu}}{\alpha\gamma}yz \\
& \quad - \frac{c_{\nu}}{\gamma}z^2 + \frac{a_{\lambda}b_{\nu}-a_{\lambda}c_{\mu}-\alpha a_{\mu}a_{\nu}-\beta d_{\lambda}}{\alpha\gamma}x + \frac{-\alpha a_{\nu}b_{\mu}+\alpha d_{\mu}+b_{\lambda}b_{\nu}-b_{\lambda}c_{\mu}}{\alpha\gamma}y\\
& \quad + \frac{-\alpha a_{\nu}c_{\mu}-\alpha d_{\nu}+b_{\nu}c_{\lambda}-c_{\lambda}c_{\mu}}{\alpha\gamma}z + \frac{-\alpha a_{\nu}d_{\mu}+b_{\nu}d_{\lambda}-c_{\mu}d_{\lambda}}{\alpha\gamma}.
\end{align*}

Now,
\begin{align*}
f_{32}x
&=\bigl(\alpha^{-1}yz-\alpha^{-1}a_{\lambda}x-\alpha^{-1}b_{\lambda}y-\alpha^{-1}c_{\lambda}z-\alpha^{-1}d_{\lambda}\bigr)x\\
&=\alpha^{-1}yzx-\alpha^{-1}a_{\lambda}x^2-\alpha^{-1}b_{\lambda}yx-\alpha^{-1}c_{\lambda}zx-\alpha^{-1}d_{\lambda}x.
\end{align*}

We first reduce $yzx$:
\begin{align*}
yzx
&=y(zx)\\
&=y(\beta xz+a_{\mu}x+b_{\mu}y+c_{\mu}z+d_{\mu})\\
&=\beta yxz+a_{\mu}yx+b_{\mu}y^2+c_{\mu}yz+d_{\mu}y.
\end{align*}

Using the rule for $yx$, we get that
\begin{align*}
yxz
&=\bigl(\gamma^{-1}xy-\gamma^{-1}a_{\nu}x-\gamma^{-1}b_{\nu}y-\gamma^{-1}c_{\nu}z-\gamma^{-1}d_{\nu}\bigr)z\\
&=\gamma^{-1}xyz-\gamma^{-1}a_{\nu}xz-\gamma^{-1}b_{\nu}yz-\gamma^{-1}c_{\nu}z^2-\gamma^{-1}d_{\nu}z,
\end{align*}

and similarly
\begin{align*}
a_{\mu}yx
&=a_{\mu}\bigl(\gamma^{-1}xy-\gamma^{-1}a_{\nu}x-\gamma^{-1}b_{\nu}y-\gamma^{-1}c_{\nu}z-\gamma^{-1}d_{\nu}\bigr).
\end{align*}

Hence
\begin{align*}
yzx
&= \frac{\beta}{\gamma}xyz
+\frac{a_{\mu}}{\gamma}xy
-\frac{\beta a_{\nu}}{\gamma}xz
+b_{\mu}y^2
+\left(c_{\mu}-\frac{\beta b_{\nu}}{\gamma}\right)yz
-\frac{\beta c_{\nu}}{\gamma}z^2\\
&
-\frac{a_{\mu}a_{\nu}}{\gamma}x
+\left(d_{\mu}-\frac{a_{\mu}b_{\nu}}{\gamma}\right)y
-\left(\frac{a_{\mu}c_{\nu}}{\gamma}+\frac{\beta d_{\nu}}{\gamma}\right)z
-\frac{a_{\mu}d_{\nu}}{\gamma}.
\end{align*}

Moreover,
\begin{align*}
-b_{\lambda}yx
&=-b_{\lambda}\bigl(\gamma^{-1}xy-\gamma^{-1}a_{\nu}x-\gamma^{-1}b_{\nu}y-\gamma^{-1}c_{\nu}z-\gamma^{-1}d_{\nu}\bigr)\\
&=-\frac{b_{\lambda}}{\gamma}xy+\frac{a_{\nu}b_{\lambda}}{\gamma}x+\frac{b_{\lambda}b_{\nu}}{\gamma}y+\frac{b_{\lambda}c_{\nu}}{\gamma}z+\frac{b_{\lambda}d_{\nu}}{\gamma},
\end{align*}

and
\begin{align*}
-c_{\lambda}zx
&=-c_{\lambda}\bigl(\beta xz+a_{\mu}x+b_{\mu}y+c_{\mu}z+d_{\mu}\bigr)\\
&=-\beta c_{\lambda}xz-a_{\mu}c_{\lambda}x-b_{\mu}c_{\lambda}y-c_{\lambda}c_{\mu}z-c_{\lambda}d_{\mu}.
\end{align*}

Therefore,
\begin{align*}
f_{32}x
& = \frac{\beta}{\alpha\gamma}xyz
-\frac{a_{\lambda}}{\alpha}x^2
+\frac{a_{\mu}-b_{\lambda}}{\alpha\gamma}xy
-\frac{\beta(a_{\nu}+\gamma c_{\lambda})}{\alpha\gamma}xz
+\frac{b_{\mu}}{\alpha}y^2 + \frac{\gamma c_{\mu}-\beta b_{\nu}}{\alpha\gamma}yz \\
& \quad - \frac{\beta c_{\nu}}{\alpha\gamma}z^2 + \frac{-\gamma a_{\mu}c_{\lambda}-a_{\mu}a_{\nu}+a_{\nu}b_{\lambda}-\gamma d_{\lambda}}{\alpha\gamma}x + \frac{-a_{\mu}b_{\nu}+b_{\lambda}b_{\nu}-\gamma b_{\mu}c_{\lambda}+\gamma d_{\mu}}{\alpha\gamma}y\\
& 
+\frac{-a_{\mu}c_{\nu}+b_{\lambda}c_{\nu}-\beta d_{\nu}-\gamma c_{\lambda}c_{\mu}}{\alpha\gamma}z + \frac{-a_{\mu}d_{\nu}+b_{\lambda}d_{\nu}-\gamma c_{\lambda}d_{\mu}}{\alpha\gamma}.
\end{align*}

Subtracting the two reduced expressions, we obtain the following:
{\small{
\begin{align*}
(zf_{21}-f_{32}x)
=&\frac{a_{\lambda}(\gamma-\beta)}{\alpha\gamma}x^2
+\frac{a_{\mu}(\alpha-1)+b_{\lambda}(1-\beta)}{\alpha\gamma}xy
+\frac{b_{\mu}(\alpha-\gamma)}{\alpha\gamma}y^2\\
&+\frac{\beta\bigl(a_{\nu}(1-\alpha)+c_{\lambda}(\gamma-1)\bigr)}{\alpha\gamma}xz
+\frac{b_{\nu}(\beta-1)+c_{\mu}(1-\gamma)}{\alpha\gamma}yz
+\frac{c_{\nu}(\beta-\alpha)}{\alpha\gamma}z^2\\
&+\frac{a_{\lambda}b_{\nu}-a_{\lambda}c_{\mu}+a_{\mu}a_{\nu}(1-\alpha)+\gamma a_{\mu}c_{\lambda}-a_{\nu}b_{\lambda}+(\gamma-\beta)d_{\lambda}}{\alpha\gamma}x\\
&+\frac{a_{\mu}b_{\nu}-\alpha a_{\nu}b_{\mu}-b_{\lambda}c_{\mu}+\gamma b_{\mu}c_{\lambda}+(\alpha-\gamma)d_{\mu}}{\alpha\gamma}y\\
&+\frac{a_{\mu}c_{\nu}-\alpha a_{\nu}c_{\mu}-b_{\lambda}c_{\nu}+b_{\nu}c_{\lambda}+(\gamma-1)c_{\lambda}c_{\mu}+(\beta-\alpha)d_{\nu}}{\alpha\gamma}z\\
&+\frac{a_{\mu}d_{\nu}-\alpha a_{\nu}d_{\mu}-b_{\lambda}d_{\nu}+b_{\nu}d_{\lambda}+\gamma c_{\lambda}d_{\mu}-c_{\mu}d_{\lambda}}{\alpha\gamma}.
\end{align*}
}}

Thus, the expression (\ref{Diamond3DBergman}) holds if and only if each coefficient above is zero, that is:
\begin{align}
a_{\lambda}(\gamma-\beta)&=0, \\
a_{\mu}(\alpha-1)+b_{\lambda}(1-\beta)&=0, \\
b_{\mu}(\alpha-\gamma)&=0, \\
a_{\nu}(1-\alpha)+c_{\lambda}(\gamma-1)&=0, \\
b_{\nu}(\beta-1)+c_{\mu}(1-\gamma)&=0,\\
c_{\nu}(\beta-\alpha)&=0, \\
a_{\lambda}b_{\nu}-a_{\lambda}c_{\mu}+a_{\mu}a_{\nu}(1-\alpha)+\gamma a_{\mu}c_{\lambda}-a_{\nu}b_{\lambda}+(\gamma-\beta)d_{\lambda}&=0, \\
a_{\mu}b_{\nu}-\alpha a_{\nu}b_{\mu}-b_{\lambda}c_{\mu}+\gamma b_{\mu}c_{\lambda}+(\alpha-\gamma)d_{\mu}&=0, \\
a_{\mu}c_{\nu}-\alpha a_{\nu}c_{\mu}-b_{\lambda}c_{\nu}+b_{\nu}c_{\lambda}+(\gamma-1)c_{\lambda}c_{\mu}+(\beta-\alpha)d_{\nu}&=0, \\
a_{\mu}d_{\nu}-\alpha a_{\nu}d_{\mu}-b_{\lambda}d_{\nu}+b_{\nu}d_{\lambda}+\gamma c_{\lambda}d_{\mu}-c_{\mu}d_{\lambda}&=0. 
\end{align}
\end{remark}

As we will see in the next section, some of these relations are required to assert the differential smoothness of the algebras of our interest.

\section{Differential and integral calculus}\label{Differentialandintegralcalculusbi-quadraticalgebras}

\begin{theorem}\label{Firsttheoremsmoothness3-dimensionalalgebras}
Let $A$ be a $3$-dimensional skew polynomial ring with PBW basis. If the conditions
\begin{align*}
    a_{\lambda} = b_{\mu} = c_{\nu} = &\ 0, \\
    c_{\mu}(\beta-1) = &\ 0,  \\
    a_{\mu}c_{\mu} = &\ 0,  \\
    b_{\nu}(\beta-1) = &\ 0,  \\
    d_{\nu}(\gamma-\beta^{-1}) = &\ a_{\nu}b_{\nu},  \\
    d_{\lambda}(\alpha-1) = &\ c_{\lambda}b_{\lambda},  \\
    d_{\nu}(\gamma-1) = &\ a_{\nu}b_{\nu}, \\
    b_{\lambda}(\beta-1) = &\ 0,  \\
    d_{\lambda}(\alpha\beta-1) = &\ b_{\lambda}c_{\lambda} \quad {\rm and}  \\
    a_{\mu}(\beta-1) = &\ 0, 
\end{align*}

hold, then $A$ is differentially smooth.
\end{theorem}

\begin{proof}
To define a first-order differential calculus with generators $dx$, $dy$ and $dz$, we require that the left $A$-action on $\Omega^{1}A$ be compatible with the right $A$-module structure via algebra automorphisms. More precisely, we assume that for every $p\in A$,
\begin{equation}\label{relbimod}
p\,dx_i = dx_i\,\nu_{x_i}(p), \quad \text{for } i=1,2,3,
\end{equation}

where each $\nu_{x_i}$ is an algebra automorphism.

Applying the differential $d$ to the defining relations (\ref{rel3skew}), we obtain that 
\begin{align*}
d(yz-\alpha zy) & =d(a_{\lambda}x+b_{\lambda}y+c_{\lambda}z+d_{\lambda}), \\
d(zx-\beta xz) & =d(a_{\mu}x+b_{\mu}y+c_{\mu}z+d_{\mu}) \ \text{and} \\
d(xy-\gamma yx) & = d(a_{\nu}x+b_{\nu}y+c_{\nu}z+d_{\nu}).
\end{align*}

By using the $\Bbbk$-linearity of $d$, 
\begin{align*}
d(yz)-\alpha\, d(zy) & =a_{\lambda}dx+b_{\lambda}dy+c_{\lambda}dz+d(d_{\lambda}), \\
d(zx)-\beta\, d(xz) & =a_{\mu}dx+b_{\mu}dy+c_{\mu}dz+d(d_{\mu}) \ \text{and} \\
d(xy)-\gamma\, d(yx) & = a_{\nu}dx+b_{\nu}dy+c_{\nu}dz+d(d_{\nu}).
\end{align*}

The Leibniz rule, expression (\ref{relbimod}) and the fact that $d$ vanishes on $\Bbbk$ imply that
\begin{align*}
dy\,z+dz\,\nu_{z}(y)-\alpha\,dz\,y-\alpha\,dy\,\nu_{y}(z) & =a_{\lambda}dx+b_{\lambda}dy+c_{\lambda}dz, \\
dz\,x+dx\,\nu_{x}(z)-\beta\,dx\,z-\beta\,dz\,\nu_{z}(x) & =a_{\mu}dx+b_{\mu}dy+c_{\mu}dz \ \text{and} \\
dx\,y+dy\,\nu_{y}(x)-\gamma\,dy\,x-\gamma\,dx\,\nu_{x}(y) & = a_{\nu}dx+b_{\nu}dy+c_{\nu}dz.
\end{align*}

After grouping by terms, these identities can be rewritten as
\begin{align*}
a_{\lambda}dx+dy\big(z-\alpha\nu_{y}(z)-b_{\lambda}\big)+dz\big(\nu_{z}(y)-\alpha y-c_{\lambda}\big) & =0, \\
dx\big(\nu_{x}(z)-\beta z-a_{\mu}\big)-b_{\mu}dy+dz\big(x-\beta\nu_{z}(x)-c_{\mu}\big) & =0 \ \text{and} \\
dx\big(y-\gamma\nu_{x}(y)-a_{\nu}\big)+dy\big(\nu_{y}(x)-\gamma x-b_{\nu}\big)-c_{\nu}dz & =0.
\end{align*}

In particular, we must have $a_{\lambda}=b_{\mu}=c_{\nu}=0$, and the remaining expressions prescribe the values of the automorphisms $\nu_x$, $\nu_y$, and $\nu_z$ on generators.

Now, we consider the following automorphisms:
\begin{align}
   \nu_{x}(x) = &\ \beta^{-1}x, & \nu_{x}(y) = &\ \gamma^{-1}(y-a_{\nu}), & \nu_{x}(z) = &\ \beta z+a_{\mu}, \label{Auto1} \\ 
    \nu_{y}(x) = &\ \gamma x+b_{\nu}, & \nu_{y}(y) = &\ y, &  \nu_{y}(z) = &\ \alpha^{-1}(z-b_{\lambda}), \label{Auto2} \\
    \nu_{z}(x) = &\ \beta^{-1}(x-c_{\mu}), & \nu_{z}(y) = &\ \alpha y+c_{\lambda}, & \nu_{z}(z) = &\ \beta z. \label{Auto3}
\end{align}

The map $\nu_{x}$ extends to an algebra homomorphism $A\to A$ if and only if the assignments in \eqref{Auto1} respect the relations {\rm (}\ref{rel3skew}{\rm )}, that is,
\begin{align*}
    \nu_{x}(y)\nu_{x}(z)-\alpha\nu_{x}(z)\nu_{x}(y) = &\ b_{\lambda}\nu_{x}(y) + c_{\lambda}\nu_{x}(z)+d_{\lambda}, \\
   \nu_{x}(z)\nu_{x}(x)-\beta\nu_{x}(x)\nu_{x}(z) = &\ a_{\mu}\nu_{x}(x) + c_{\mu}\nu_{x}(z)+d_{\mu} \ \text{and} \\
   \nu_{x}(x)\nu_{x}(y)-\gamma\nu_{x}(y)\nu_{x}(x) = &\ a_{\nu}\nu_{x}(x) + b_{\nu}\nu_{x}(y)+d_{\nu}.
\end{align*}

Equivalently, this yields the new conditions
\begin{align*}
     c_{\mu}(\beta-1) = &\ 0,  \\
     c_{\mu}a_{\mu}= &\ 0, \\
     b_{\nu}(\beta-1)= &\ 0 \quad \text{and} \\
     d_{\nu}(\gamma-\beta^{-1}) = &\ a_{\nu}b_{\nu}.
\end{align*}

In a similar way, $\nu_{y}$ extends to an algebra homomorphism of $A$ if and only if the equalities
\begin{align*}
     d_{\lambda}(\alpha-1) = &\ c_{\lambda}b_{\lambda} \quad \text{and} \\
     d_{\nu}(\gamma-1) = &\ a_{\nu}b_{\nu}
\end{align*}

hold. Finally, requiring that $\nu_{z}$ extend to an algebra homomorphism gives the additional conditions
\begin{align*}
    b_{\lambda}(\beta-1) = &\ 0 \quad \text{and}\\
    a_{\mu}(\beta-1) = &\ 0.
\end{align*}

To impose that the automorphisms commute, that is, 
\begin{align}\label{commuauto}
   \nu_{x} \circ \nu_{y} = &\ \nu_{y} \circ \nu_{x}, \\
   \nu_{x} \circ \nu_{z} = &\  \nu_{z} \circ \nu_{x} \quad \text{and} \\
   \nu_{z} \circ \nu_{y} = &\ \nu_{y} \circ \nu_{z},
\end{align}

we need to verify these equalities on the generators $x$, $y$, and $z$. In particular,
\begin{align}
\nu_{x} \circ \nu_{y}(x) = &\ \gamma \beta^{-1}x+b_{\nu}, \label{comp111} \\
\nu_{y} \circ \nu_{x}(x) = &\ \gamma\beta^{-1}x+\beta^{-1}b_{\nu}, \label{comp111'} \\  
\nu_{x} \circ \nu_{y}(y) = &\ \gamma^{-1}(y-a_{\nu}), \\
\nu_{y} \circ \nu_{x}(y) = &\ \gamma^{-1}(y-a_{\nu}), \label{comp12} \\
 \nu_{x} \circ \nu_{y}(z) = &\ \alpha^{-1}(\beta z+a_{\mu}-b_{\lambda}) \quad \text{and} \label{comp123} \\
 \nu_{y} \circ \nu_{x}(z) = &\ \alpha^{-1}(\beta z -\beta b_{\lambda})+a_{\mu}. \label{comp213}
\end{align}

Comparing \eqref{comp111} with \eqref{comp111'} and \eqref{comp123} with \eqref{comp213}, these pairs coincide if and only if the equalities
$$
b_{\nu}(1-\beta)=0 \quad {\rm and}\quad a_{\mu}(1-\alpha)=b_{\lambda}(1-\beta)
$$

hold. Note that 
\begin{align}
    \nu_{x} \circ \nu_{z}(x) = &\ \beta^{-2}x-\beta^{-1}c_{\mu}, \label{comp13}\\
    \nu_{z} \circ \nu_{x}(x) = &\ \beta^{-2}(x-c_{\mu}), \label{comp21} \\
    \nu_{x} \circ \nu_{z}(y) = &\ \gamma^{-1}\alpha (y- a_{\nu})+c_{\lambda}, \label{comp13'}\\
    \nu_{z} \circ \nu_{x}(y) = &\ \gamma^{-1}(\alpha y +c_{\lambda} - a_{\nu}), \label{comp22} \\
    \nu_{x} \circ \nu_{z}(z) = &\ \beta^2z+\beta a_{\mu} \quad \text{and} \label{comp22'} \\
    \nu_{z} \circ \nu_{x}(z) = &\ \beta^2z+a_{\mu},  \label{comp23}
\end{align}

and the pairs \eqref{comp13}--\eqref{comp21}, \eqref{comp13'}--\eqref{comp22} and \eqref{comp22'}--\eqref{comp23} coincide provided that 
$$
c_{\mu}(\beta-1)=0,\quad c_{\lambda}(1-\gamma)=a_{\nu}(1-\alpha) \quad {\rm and} \quad a_{\mu}(\beta-1)=0.
$$

Finally, 
\begin{align}
    \nu_{z} \circ \nu_{y}(x) = &\ \gamma\beta^{-1}x-\gamma\beta^{-1}c_{\mu}+b_{\nu}, \label{comp31'} \\
    \nu_{y} \circ \nu_{z}(x) = &\ \beta^{-1}(\gamma x+b_{\nu}-c_{\mu}), \label{comp31} \\
    \nu_{z} \circ \nu_{y}(y) = &\ \alpha y+c_{\lambda}, \label{comp32'} \\
    \nu_{y} \circ \nu_{z}(y) = &\ \alpha y+c_{\lambda},  \label{comp32} \\
    \nu_{z} \circ \nu_{y}(z) = &\ \alpha^{-1}(\beta z-b_{\lambda})  \quad \text{and} \label{comp33'}\\
    \nu_{y} \circ \nu_{z}(z) = &\ \beta\alpha^{-1}(z-b_{\lambda}). \label{comp33}
\end{align}

Thus, \eqref{comp31'} and \eqref{comp31}, as well as \eqref{comp33'} and \eqref{comp33}  coincide whenever
$$
b_{\nu}(1-\beta)=c_{\mu}(1-\gamma) \quad {\rm and} \quad b_{\lambda}(\beta-1) = 0.
$$

Let $\Omega^{1}A$ be a free right $A$-module of rank three with generators $dx$, $dy$, and $dz$. For each $p\in A$, define a left $A$-module structure by
\begin{align}
    p\,dx = &\ dx \nu_{x}(p), \notag \\ 
    p\,dy = &\ dy\nu_{y}(p)\quad  \text{and} \notag \\
    p\,dz = &\ dz\nu_{z}(p). \label{relrightmod}
\end{align}

Then the bimodule relations in $\Omega^{1}A$ are
\begin{align}
x\,dx = &\ dx \beta^{-1}x, \notag \\
x\,dy = &\ dy(\gamma x+b_{\nu}), \notag \\
x\,dz = &\ dz\beta^{-1}(x-c_{\mu}), \label{rel1} \\
y\,dx = &\ dx\gamma^{-1}(y-a_{\nu}), \notag \\  
y\,dy = &\ dy\,y, \notag \\
y\,dz = &\ dz(\alpha y+c_{\lambda}), \label{rel2} \\
z\,dx = &\ dx(\beta z+a_{\mu}), \notag \\
z\,dy = &\ dy\alpha^{-1}(z-b_{\lambda}) \quad \text{and} \notag \\
z\,dz = &\ dz\beta z. \label{rel3} 
\end{align}

Our aim is to extend the assignment
\[
x \mapsto dx, \qquad y \mapsto dy  \quad {\rm and} \quad z\mapsto dz
\]

to a $\Bbbk$-linear map $d: A \to \Omega^{1} A$ satisfying the Leibniz rule. This is possible precisely when $d$ is compatible with the nontrivial relations {\rm (}\ref{rel3skew}{\rm )}, i.e.\ when
\begin{align*}
d(yz)-\alpha\, d(zy) & =b_{\lambda}dy+c_{\lambda}dz+d(d_{\lambda}), \\
d(zx)-\beta\, d(xz) & =a_{\mu}dx+c_{\mu}dz+d(d_{\mu}) \ \text{and} \\
d(xy)-\gamma\, d(yx) & = a_{\nu}dx+b_{\nu}dy+d(d_{\nu})
\end{align*}
hold.

Define $\Bbbk$-linear maps
\[
\partial_{x}, \partial_{y}, \partial_{z}: A \rightarrow A
\]

by requiring that
\[
d(a)=dx\,\partial_{x}(a)+dy\,\partial_{y}(a)+dz\,\partial_{z}(a), \quad \text{for all } a \in A.
\]

Since $dx$, $dy$ and $dz$ form a free basis of the right $A$-module $\Omega^1A$, the maps $\partial_x,\partial_y,\partial_z$ are well-defined. Moreover, 
$$
d(a)=0 \quad {\rm if\, and\, only\, if} \quad \partial_{x}(a) = \partial_{y}(a) = \partial_{z}(a)=0.
$$

Using the relations in {\rm (}\ref{relrightmod}{\rm )} together with the definitions of $\nu_{x}$, $\nu_{y}$, and $\nu_{z}$, we obtain that 
\begin{align}
\partial_{x}(x^ky^lz^s) = &\ \sum_{i=0}^{k-1}\beta^{-i}x^{k-1}y^lz^s, \notag \\
\partial_{y}(x^ky^lz^s) = &\ l(\gamma x+b_{\nu})^ky^{l-1}z^s \quad \text{and}  \\
\partial_{z}(x^ky^lz^s) = &\ \sum_{i=0}^{s-1}\beta^{i-k}(x-c_{\mu})^k(\alpha y+c_{\lambda})^lz^{s-1}. \notag
\end{align}

Therefore $d(a)=0$ if and only if $a$ is a scalar multiple of the identity, showing that $(\Omega A,d)$ is connected, where 
$$
\Omega A = \Omega^0 A \oplus \Omega^1 A \oplus \Omega^2 A.
$$

The universal extension of $d$ to higher forms, compatible with {\rm (}\ref{rel1}{\rm )}, {\rm (}\ref{rel2}{\rm )} and {\rm (}\ref{rel3}{\rm )}, yields the following relations in $\Omega^2A$:
\begin{align}
dy\wedge dx = &\ -\gamma^{-1}dx\wedge dy, \\
dz\wedge dx = &\ -\beta\, dx\wedge dz \quad \text{and} \\ 
dz\wedge dy = &\ -\alpha^{-1}dy\wedge dz. \label{relsecond}
\end{align}
Since $\nu_{x}$, $\nu_{y}$, and $\nu_{z}$ commute, no additional relations arise, and hence
\[
\Omega^2A =  dx\wedge dy\,A\oplus dx\wedge dz\,A\oplus dy\wedge dz\,A.
\]

Finally, since $\Omega^3A = \omega A\cong A$ as a right and left $A$-module with $\omega=dx\wedge dy \wedge dz$ and $\nu_{\omega}=\nu_{x}\circ\nu_{y}\circ\nu_{z}$, the form $\omega$ is a volume form for $A$. By Proposition \ref{BrzezinskiSitarz2017Lemmas2.6and2.7} (2), $\omega$ is an integral form upon taking
\begin{align*}
\omega_1^1  = &\ \bar{\omega}_1^1 = dx, \\  
\omega_2^1 = &\ \bar{\omega}_2^1 = dy, \\
\omega_3^1 = &\ \bar{\omega}_3^1 = dz, \\
\omega_1^2 = &\ dy\wedge dz, \\
\omega_2^2 = &\ -\gamma dx\wedge dz, \\ 
\omega_3^2 = &\ \alpha\beta^{-1}dx\wedge dz, \\
\bar{\omega}_1^2 = &\ \gamma \beta^{-1}dy\wedge dz, \\
\bar{\omega}_2^2 = &\ -\alpha dx\wedge dy \quad \text{and} \\ 
\bar{\omega}_3^2 = &\ dx\wedge dy.
\end{align*}

Indeed, for $\omega' := dx\,a + dy\,b + dz\, c$ with $a, b, c \in \Bbbk$, Proposition \ref{BrzezinskiSitarz2017Lemmas2.6and2.7} (2) gives
\begin{align*}
    \sum_{i=1}^{3}\omega_{i}^{1}\pi_{\omega}(\bar{\omega}_i^{2}\wedge \omega') = &\ dx\,\pi_{\omega}(\gamma\beta^{-1}a\,dy\wedge dz\wedge dx) \\
    & + dy\,\pi_{\omega}(-\alpha b\,dx\wedge dz\wedge dy) \\
    & + dz\,\pi_{\omega}(c\,dx\wedge dy\wedge dz) \\
    = &\ dx\,a+dy\,b+dz\,c = \omega'.
\end{align*}

On the other hand, if 
$$
\omega'' := dx\wedge dy\, a+dx\wedge dz\,b+dy\wedge dz\,c \quad {\rm with} \, \, a, b, c\in \Bbbk,
$$ 

then
\begin{align*}
\sum_{i=1}^{3}\omega_{i}^{2}\pi_{\omega}(\bar{\omega}_i^{1}\wedge \omega'') = &\  dy\wedge dz\,\pi_{\omega}(c\,dx\wedge dy \wedge dz) \\
&\ - \gamma dx\wedge dz\,\pi_{\omega}(b\,dy\wedge dx\wedge dz) \\
    &\ + \alpha\beta^{-1}dx\wedge dy\,\pi_{\omega}(a\,dz \wedge dx \wedge dy) \\ 
    = &\ dx\wedge dy\,a+dx\wedge dz\,b+dy\wedge dz\,c = \omega''.
\end{align*}

As was shown above, all forms of degrees $1$ and $2$ are generated by $\omega_i^j$ and $\bar{\omega}_i^{3-j}$ for $j=1,2$ and $i=1,2,3$. Therefore, Proposition \ref{BrzezinskiSitarz2017Lemmas2.6and2.7} (2) ensures that $\omega$ is an integral form. Finally, Proposition \ref{integrableequiva} implies that $(\Omega A, d)$ is an integrable differential calculus of degree $3$. Since ${\rm GKdim}(A) = 3$ \cite[Section 8.7]{Fajardoetal2020}, it follows that $A$ is differentially smooth.
\end{proof}

\begin{theorem}\label{Secondtheoremsmoothness3-dimensionalalgebras}
If one of the conditions $a_{\lambda}\not = 0$, $b_{\mu}\not=0$ or $c_{\nu}\not=0$ holds, then there are no one-dimensional connected integrable calculi over $A$.
\end{theorem}
\begin{proof}
By contradiction. Suppose that $A$ has a first order differential calculus $(\Omega A, d)$. Without loss of generality, we consider the case $a_{\lambda}\not =0$. Since $d$ must be compatible with the relations {\rm (}\ref{rel3skew}{\rm )}, then we get that
\begin{equation*}
    dy z-\alpha ydz = a_{\lambda}dx +b_{\lambda}dy + c_{\lambda}dz,
\end{equation*}
whence $dx$ is generated by $dy$ and $dz$. This means that $\Omega^1A$ is generated by two elements and $\Omega^3A=\Omega^1A\wedge\Omega^1A\wedge\Omega^1A=0$, i.e. there is no third-order calculus. Since ${\rm GKdim}(A) = 3$, the assertion follows.
\end{proof}

Table \ref{table1} contains the parameter values for each one of the $15$ algebras. By Theorems \ref{Firsttheoremsmoothness3-dimensionalalgebras} and \ref{Secondtheoremsmoothness3-dimensionalalgebras}, the symbols $\checkmark$ and $\star$ denote a positive and a negative answer, respectively, on differential smoothness of $3$-dimensional skew polynomial algebras.

\begin{table}[H]
\caption{Parameter values for $3$-dimensional skew polynomial algebras}
\label{table1}
\centering
\resizebox{\textwidth}{!}{%
\begin{tabular}{|*{17}{c|}}
\hline
 & $\alpha$ & $\beta$ & $\gamma$ & $a_{\lambda}$ & $b_{\lambda}$ & $c_{\lambda}$ & $d_{\lambda}$ & $a_{\mu}$ & $b_{\mu}$ & $c_{\mu}$ & $d_{\mu}$ & $a_{\nu}$ & $b_{\nu}$ & $c_{\nu}$ & $d_{\nu}$ & D.S. \\ \hline\hline
(1) & $\alpha$ & $\beta$ & $\gamma$ & 0 & 0 & 0 & 0 & 0 & 0 & 0 & 0 & 0 & 0 & 0 & 0 & $\checkmark$ \\ \hline
(2)(i) & $1$ & $\beta$ & $1$ & 0 & 0 & 1 & 0 & 0 & 1 & 0 & 0 & 1 & 0 & 0 & 0 & $\star$ \\ \hline
(2)(ii) &  $1$ & $\beta$ & $1$  & 0 & 0 & 1 & 0 & 0 & 0 & 0 & $b$ & 1 & 0 & 0 & 0 & $\checkmark$ \\ \hline
(2)(iii) &  $1$ & $\beta$ & $1$  & 0 & 0 & 0 & 0 & 0 & 1 & 0 & 0 & 0 & 0 & 0 & 0 & $\star$ \\ \hline
(2)(iv) &  $1$ & $\beta$ & $1$  & 0 & 0 & 0 & 0 & 0 & 0 & 0 & $b$ & 0 & 0 & 0 & 0 & $\checkmark$ \\ \hline
(2)(v) &  $1$ & $\beta$ & $1$  & 0 & 0 & $a$ & 0 & 0 & 0 & 0 & 0 & 1 & 0 & 0 & 0 & $\checkmark$ \\ \hline
(2)(vi) &  $1$ & $\beta$ & $1$  & 0 & 0 & 1 & 0 & 0 & 0 & 0 & 0 & 0 & 0 & 0 & 0 & $\checkmark$ \\ \hline
(3)(i) & $\alpha$ & $\beta$ & $\alpha$ & 0 & 0 & 0 & 0 & 0 & 1 & 0 & $b$ & 0 & 0 & 0 & 0 & $\star$ \\ \hline
(3)(ii) & $\alpha$ & $\beta$ & $\alpha$ & 0 & 0 & 0 & 0 & 0 & 0 & 0 & $b$ & 0 & 0 & 0 & 0 & $\checkmark$ \\ \hline
(4) & $\alpha$ & $\alpha$ & $\alpha$ & $a_1$ & 0 & 0 & $b_1$ & 0 & $a_2$ & 0 & $b_2$ & 0 & 0 & $a_3$ & $b_3$ & $\star$ \\ \hline
(5)(i) & $1$ & $1$ & $1$ & 1 & 0 & 0 & 0 & 0 & 1 & 0 & 0 & 0 & 0 & 1 & 0 & $\star$ \\ \hline
(5)(ii) &  $1$ & $1$ & $1$  & 0 & 0 & 0 & 0 & 0 & 0 & 0 & 0 & 0 & 0 & 1 & 0 & $\star$ \\ \hline
(5)(iii) &  $1$ & $1$ & $1$  & 0 & 0 & 0 & 0 & 0 & 0 & 0 & 0 & 0 & 0 & 0 & $b$ & $\checkmark$ \\ \hline
(5)(iv) &  $1$ & $1$ & $1$  & 0 & -1 & 0 & 0 & 1 & 1 & 0 & 0 & 0 & 0 & 0 & 0 & $\star$ \\ \hline
(5)(v) &  $1$ & $1$ & $1$  & 0 & 0 & $a$ & 0 & 0 & 0 & 1 & 0 & 0 & 0 & 0 & 0 & $\checkmark$ \\ \hline
\end{tabular}%
}
\end{table}

\begin{remark}
In \cite{ReyesSarmiento2022}, the differential smoothness of the algebra of type (5)(v) cannot be settled by the method developed therein. This arises from the fact that the defining relations for case (5)(v) are taken from Rosenberg's statement in \cite[Theorem C4.3.1, case (e)(v), p. 101]{Rosenberg1995}. As explained in Remark \ref{Rosenberg}, that statement contains a typo: the relation $zx-xz=x$ must be replaced with $zx-xz=z$. Consequently, the arguments in \cite{ReyesSarmiento2022} that rely on the misprinted presentation do not apply to the genuine algebra of type (5)(v); using the corrected relations, its differential smoothness is settled in the present work.
\end{remark}

\section{Future work}\label{Futurework}

Jordan \cite{Jordan1993} introduced {\em certain iterated skew polynomial rings} $K[x; \sigma][y; \sigma^{-1}, \delta]$ in two indeterminates over a commutative ring $K$ which include the universal enveloping algebra of the Lie algebra $\mathfrak{sl}(2)$ and several quantum groups. Two years later, in \cite{Jordan1995} he extended this class of rings by considering an extra parameter $\rho \in \Bbbk \ \backslash \ \{0\}$ and include the quantized Weyl algebra, the universal enveloping algebra of the Dispin Lie superalgebra, and an algebra introduced by Woronowicz \cite{Woronowicz1987}. Let us see the details.

\begin{definition}[{\cite[Section 1.1]{Jordan1995}}]\label{Jordan1995Section1.1}
Let $K$ be a finitely generated commutative algebra over an algebraically closed field $\Bbbk$, let $\sigma$ be a $\Bbbk$-automorphism of $K$, and consider $u\in K$ and $\rho \in \Bbbk \ \backslash \ \{0\}$. Form the skew polynomial ring of automorphism type $K[x;\sigma]$ and extend $\sigma$ to $K[x;\sigma]$ by setting 
$$
\sigma(x) = \rho^{-1}x \quad {\rm whence} \quad \sigma^{-1}(x) = \rho x.
$$

There is an $\sigma^{-1}$-derivation $\delta$ of $K[x;\sigma]$ such that $\delta(K) = 0$ and $\delta(x) = u - \rho \sigma(u)$ \cite[p. 41]{Cohn1985} (see also \cite[Section 2.8]{GoodearlLetzter1994}). Consider the skew polynomial ring
\begin{align*}
K[x;\sigma][y;\sigma^{-1},\delta], 
\end{align*}

with defining relations given by 
\begin{align}
    xk = &\, \sigma(k)x, \\ 
    yk = &\, \sigma^{-1}(k)y, \quad {\rm for}\ k\in K, \, {\rm and} \\
    yx = &\, \sigma^{-1}(x)y + \delta(x) = \rho xy + u - \rho \sigma(u).
\end{align}
\end{definition}

He stated \cite[Section 1]{Jordan1995} that not all noncommutative polynomial $\Bbbk$-algebras in three indeterminates considered by Bell and Smith \cite{BellSmith1990} are iterated skew polynomial rings over $\Bbbk$, but these can be obtained from the construction in Definition \ref{Jordan1995Section1.1} with $K = \Bbbk[t]$. Also, he mentioned that their objects are related with the class of iterated skew polynomial rings in two indeterminates defined by Goodearl and Letzter \cite{GoodearlLetzter1994}, called $q$-{\em skew polynomial rings}, so that a natural task is to investigate the differential smoothness of both families of skew polynomial rings.

\section{Declarations}

The authors have no conflict of interest to disclose.

\end{document}